\documentclass[letterpaper,11pt]{amsart}

\usepackage[margin=1.2in]{geometry}
\usepackage{amsmath,amsthm,amssymb}
\usepackage{xspace,xcolor}
\usepackage[breaklinks,colorlinks,citecolor=teal,linkcolor=teal,urlcolor=teal,pagebackref,hyperindex]{hyperref}
\usepackage[alphabetic]{amsrefs}
\usepackage[all]{xy}
\usepackage{color}

\theoremstyle{plain}
\newtheorem{thm}{Theorem}[section]

\newtheorem{cor}[thm]{Corollary}
\newtheorem*{conj}{Conjecture}

\theoremstyle{definition}

\theoremstyle{remark}
\newtheorem{rmk}[thm]{Remark}

\def\Z{{\mathbf Z}}
\def\Q{{\mathbf Q}}

\def\C{{\mathbf C}}

\def\A{{\mathbf A}}

\def\cO{\mathcal{O}}

\def\fa{\mathfrak{a}}

\def\.{\cdot}
\def\^{\widehat}

\def\({\left(}
\def\){\right)}

\renewcommand{\and}{ \ \ \text{ and } \ \ }

\DeclareMathOperator{\lct} {lct}

\begin{document}

\title{Bernstein-Sato polynomials for general ideals vs. principal ideals}
\author[M. Musta\c{t}\u{a}]{Mircea~Musta\c{t}\u{a}}
\address{Department of Mathematics, University of Michigan, 530 Church Street,
Ann Arbor, MI 48109, USA}
\email{{\tt mmustata@umich.edu}}

\thanks{The author was partially supported by NSF grant DMS-1701622 and a Simons Fellowship.}


\begin{abstract}
We show that given an ideal $\fa$ generated by regular functions $f_1,\ldots,f_r$ on $X$,
the Bernstein-Sato polynomial of $\fa$ is equal to the reduced Bernstein-Sato polynomial of
the function $g=\sum_{i=1}^rf_iy_i$ on $X\times\A^r$. By combining this with results from \cite{BMS}, we 
relate invariants and properties of $\fa$ to those of $g$. We also use the result on
Bernstein-Sato polynomials to
show that
the Strong Monodromy Conjecture for Igusa zeta functions of principal ideals implies a similar statement for arbitrary ideals.
\end{abstract}

\maketitle

\section{Introduction}

Given a smooth complex algebraic variety $X$ and a nonzero regular function $f\in\cO_X(X)$, the \emph{Bernstein-Sato polynomial} $b_f(s)\in\C[s]$
is the monic polynomial of minimal degree such that 
$$
b_f(s)f^s\in {\mathcal D}_X[s]\bullet f^{s+1}.
$$
Here ${\mathcal D}_X$ is the sheaf of differential operators on $X$ and we use $\bullet$ to denote the action of differential operators.
Note that $f^s$ can be treated as a symbol on which differential operators
act in the expected way. By making $s=-1$, we see that if $f$ is not invertible, then $b_f(s)$ is divisible by $(s+1)$, and the quotient
$\widetilde{b}_f(s)=b_f(s)/(s+1)$ is the \emph{reduced Bernstein-Sato} polynomial of $f$. The existence of $b_f(s)$ was proved by Bernstein
for the case when $X=\A^n$ in
\cite{Bernstein} and a proof in the general case (in the analytic setting) is given in \cite{Bjork}. The Bernstein-Sato polynomial of $f$ is a subtle invariant
of the singularities of the hypersurface defined by $f$ and it is connected to several other invariants of singularities (for example, by
\cite{Malgrange}, its roots determine the
eigenvalues of the monodromy action on the cohomology of the Milnor fiber). 

The above invariant has been extended to arbitrary (nonzero) coherent ideals $\fa$ in $\cO_X$ in \cite{BMS}. Working locally, we may and will assume that
we have nonzero regular functions $f_1,\ldots,f_r\in\cO_X(X)$ that generate the ideal $\fa$. In this case, the Bernstein-Sato polynomial $b_{\fa}(s)\in\C[s]$
is the monic polynomial of minimal degree such that
$$
b_{\fa}(s)f_1^{s_1}\cdots f_r^{s_r}\in \sum_{|u|=1}{\mathcal D}_X[s_1,\ldots,s_r]\bullet \prod_{u_i<0}{{s}\choose {-u_i}}f_1^{s_1+u_1}\cdots f_r^{s_r+u_r},
$$
where the sum is over all $u=(u_1,\ldots,u_r)\in\Z^r$ such that $|u|:=\sum_iu_i=1$. Here $s=s_1+\ldots+s_r$, where $s_1,\ldots,s_r$ are independent variables,
$f_1^{s_1}\cdots f_r^{s_r}$ is a symbol on which differential operators act in the expected way, and for every positive integer $m$, we put
${{s_i}\choose {m}}=\frac{1}{m!}\prod_{j=0}^{m-1}(s_i-j)$. The existence, independence of the choice of the generators $f_1,\ldots,f_r$, and some basic properties
of $b_{\fa}(s)$ were proved in \cite{BMS}. The main observation
 of this note is the following result. Given $f_1,\ldots,f_r$ as above, we consider the regular function
$g=\sum_{i=1}^rf_iy_i$ on $X\times\A^r$, where $y_1,\ldots,y_r$ are the coordinates on $\A^r$.

\begin{thm}\label{thm_description}
If $f_1,\ldots,f_r$ are nonzero regular functions on the smooth, complex algebraic variety $X$, generating the coherent ideal $\fa$, and if 
$g=\sum_{i=1}^rf_iy_i$, then $b_{\fa}(s)=\widetilde{b}_g(s)$.
\end{thm}

In fact, this observation can be used to give a new proof of the existence of $b_{\fa}(s)$ and of its independence of the generators $f_1,\ldots,f_r$.
We hope that it will be useful for extending properties of Bernstein-Sato polynomials from the case of principal ideals to arbitrary ones. 

By combining the above description of $b_{\fa}(s)$ with results in \cite{BMS}, we can relate invariants and properties of $g$ with those
of the ideal $\fa$. Recall that by a result of Kashiwara \cite{Kashiwara}, for every nonzero $f\in\cO_X(X)$,
all roots of the Bernstein-Sato polynomial $b_f(s)$ are negative rational numbers. If $f$ is not invertible, then the negative of the largest root 
of $\widetilde{b}_f(s)$ is the \emph{minimal exponent} $\widetilde{\alpha}_f$ of $f$ (with the convention that $\widetilde{\alpha}_f=\infty$ if 
$\widetilde{b}_f(s)=1$, which is the case if and only if the hypersurface defined by $f$ is smooth). Therefore $\min\{1,\widetilde{\alpha}_f\}$
is the negative of the largest root of $b_f(s)$; by a result of Lichtin and Koll\'{a}r (see \cite[Theorem~10.6]{Kollar}), this is equal to the log canonical threshold
${\rm lct}(f)$ of $f$.

\begin{cor}\label{cor1}
With the notation in the theorem, we have $\widetilde{\alpha}_g={\rm lct}(\fa)$.
\end{cor}

\begin{cor}\label{cor2}
With the notation in the theorem, if $\fa$ defines a reduced, complete intersection subscheme $W$, of pure codimension $r$, then
$W$ has rational singularities if and only if $\widetilde{\alpha}_g=r$ and $-r$ is a root of multiplicity $1$ of $\widetilde{b}_g(s)$.
\end{cor}

Finally, we apply the description of $b_{\fa}(s)$ in the theorem to show that the Strong Monodromy Conjecture for Igusa zeta functions 
associated to hypersurfaces
implies the similar statement for arbitrary ideals. For the sake of simplicity, we work in the $p$-adic setting, though a similar result holds for the
motivic zeta function (see Remark~\ref{rmk_motivic} below).

Recall that if $f\in\Z_p[x_1,\ldots,x_n]$ is a nonzero polynomial over the ring of $p$-adic integers, the Igusa zeta function associated to $f$ is the
formal power series in $p^{-s}$ given by
$$Z_p(f;s):=\int_{\Z_p^n}|f(x)|_p^sd\mu_p(x),$$
where $|\cdot|_p$ is the $p$-adic absolute value on $\Q_p$ and $\mu_p$ is the Haar measure on $\Q_p^n$. This power series encodes
the numbers $a_m$ of roots of $f$ in $(\Z/p^m\Z)^n$ for $m\geq 1$. It was shown by Igusa \cite{Igusa1}, \cite{Igusa2} that $Z_p(f; s)$ is a rational function of $p^{-s}$,
with the candidate poles determined in terms of a log resolution of the pair $(\A_{\C}^n,f)$.
The following is the outstanding open problem in this area:

\begin{conj}[Strong Monodromy Conjecture, Igusa]
Given $f\in\Z[x_1,\ldots,x_n]$, for every prime $p$ large enough, if
 $s_0$ is a pole of $Z_p(f;s)$, then ${\rm Re}(s_0)$ is a root of $b_f(s)$. Moreover, if the order of $s_0$ as a pole is $m$,
 then ${\rm Re}(s_0)$ is a root of $b_f(s)$ of multiplicity $\geq m$.
\end{conj}

One can study an analogue of Igusa's zeta function for arbitrary ideals $\fa\subseteq\Z_p[x_1,\ldots,x_n]$ (see \cite{Veys}). 
More precisely, if $f_1,\ldots,f_r$ generate $\fa$, then we have a function 
$\varphi_{\fa}\colon \Z_p^n\to\Q$ given by $\varphi_{\fa}(x)=\max_{i=1}^n|f_i(x)|_p$ and the corresponding Igusa zeta function
$$Z_p(\fa;s):=\int_{\Z_p^n}\varphi_{\fa}(x)^sd\mu_p(x).$$
Again, this is a rational function of $p^{-s}$ and candidate poles can be given in terms of a log resolution of $(\A_{\C}^n,\fa)$.

\begin{thm}\label{thm_Igusa}
If $\fa$ is the ideal of $\Z_p[x_1,\ldots,x_n]$ generated by the nonzero polynomials $f_1,\ldots,f_r$
and if $g=\sum_{i=1}^rf_iy_i\in\Z_p[x_1,\ldots,x_n,y_1,\ldots,y_r]$, then 
$$Z_p(g;s)=\frac{1-p^{-1}}{1-p^{-s-1}}Z_p(\fa;s).$$
In particular, if $f_1,\ldots,f_r\in\Z[x_1,\ldots,x_n]$ and
$g$ satisfies the Strong Monodromy Conjecture, then for every prime $p$ large enough,
if $s_0$ is a pole of $Z_p(\fa,s)$ of order $m$, then ${\rm Re}(s_0)$ is a root of $b_{\fa}(s)$ of multiplicity $\geq m$.
\end{thm}

In the next section we give the proof of Theorem~\ref{thm_description} and of its corollaries. The last section contains the proof of Theorem~\ref{thm_Igusa}.

\subsection*{Acknowledgement} I am indebted to Nero Budur for bringing to my attention the reference \cite{Javier} and to Wim Veys for pointing out
an inaccuracy in a previous version of this note.

\section{The description of the Bernstein-Sato polynomial of an ideal}

We begin with the formula relating the Bernstein-Sato polynomials of $\fa$ and $g$.

\begin{proof}[Proof of Theorem~\ref{thm_description}]
By taking an affine open cover of $X$, we see that we may assume that $X={\rm Spec}(R)$ is affine.
By definition, the Bernstein-Sato polynomial $b_g(s)$ is the monic polynomial of minimal degree such that there is $P\in \Gamma(X\times\A^r,
{\mathcal D}_{X\times\A^r})[s]$ such that 
\begin{equation}\label{eq_proof1}
b_g(s)g^s=P\bullet g^{s+1}.
\end{equation}
Such $P$ can be uniquely written as $P=\sum_{\alpha,\beta\in\Z_{\geq 0}^r}P_{\alpha,\beta}\frac{1}{\beta !}y^{\alpha}\partial_y^{\beta}$,
with $P_{\alpha,\beta}\in\Gamma(X,{\mathcal D}_X)[s]$, only finitely many being nonzero. 
Here we use the multi-index notation $y^{\alpha}=y_1^{\alpha_1}\cdots y_r^{\alpha_r}$
and $\partial_y^{\beta}=\partial_{y_1}^{\beta_1}\cdots\partial_{y_r}^{\beta_r}$ and $\beta!=\prod_{i=1}^r(\beta_i)!$
for $\alpha=(\alpha_1,\ldots,\alpha_r)$ and $\beta=(\beta_1,\ldots,\beta_r)$
in $\Z_{\geq 0}^r$. Furthermore, the equality in (\ref{eq_proof1}) is equivalent to
\begin{equation}\label{eq_proof2}
b_g(m)g^m=\sum_{\alpha,\beta}P_{\alpha,\beta}(m)\bullet g^{m+1}\quad\text{for all}\quad m\geq 0.
\end{equation}
Since $g=\sum_{i=1}^rf_iy_i$, we have
\begin{equation}\label{eq_proof3}
b_g(m)g^m=b_g(m)\cdot \sum_{|a|=m}{{m}\choose{a_1,\ldots,a_r}}f_1^{a_1}\cdots f_r^{a_r}y_1^{a_1}\cdots y_r^{a_r},
\end{equation}
where the sum is over all $a=(a_1,\ldots,a_r)\in\Z_{\geq 0}^r$ with $|a|:=a_1+\ldots+a_r=m$. 
On the other hand, the right-hand side of (\ref{eq_proof2}) is equal to 
$$
\sum_{\alpha,\beta}P_{\alpha,\beta}(m)\frac{1}{\beta !} y^{\alpha}\partial_y^{\beta}\bullet
\sum_{|b|=m+1}{{m+1}\choose {b_1,\ldots,b_r}}f_1^{b_1}\cdots f_r^{b_r}y_1^{b_1}\cdots
y_r^{b_r},
$$
where the second sum is over all $b=(b_1,\ldots,b_r)\in\Z_{\geq 0}^r$, with $|b|=m+1$.
This is further equal to
\begin{equation}\label{eq_proof4}
\sum_{\alpha,\beta}\sum_{|b|=m+1}\big(P_{\alpha,\beta}(m)\bullet f_1^{b_1}\cdots f_r^{b_r}\big)\cdot {{m+1}\choose {b_1,\ldots,b_r}}
\cdot \prod_{i=1}^r{{b_i}\choose{\beta_i}}\cdot \prod_{i=1}^ry_i^{b_i-\beta_i+\alpha_i},
\end{equation}
where we make the convention that ${{b_i}\choose{\beta_i}}=0$ if $\beta_i>b_i$. 
Via the formulas in (\ref{eq_proof3}) and (\ref{eq_proof4}), the equality in (\ref{eq_proof2}) is equivalent to the fact that for every
$a=(a_1,\ldots,a_r)\in\Z_{\geq 0}^r$, we have 
$$
b_g\big(|a|\big){{|a|}\choose {a_1,\ldots,a_r}}f_1^{a_1}\cdots f_r^{a_r}=
$$
$$
\sum_{|\beta|-|\alpha|=1}\big(P_{\alpha,\beta}\big(|a|\big)\bullet f_1^{a_1+\beta_1-\alpha_1}\cdots
f_r^{a_r+\beta_r-\alpha_r}\big)\cdot{{|a|+1}\choose {a_1+\beta_1-\alpha_1,\ldots,a_r+\beta_r-\alpha_r}}\cdot\prod_{i=1}^r{{a_i+\beta_i-\alpha_i}
\choose{\beta_i}}.
$$
An easy computation shows that this is further equivalent to
$$b_g\big(|a|\big)f_1^{a_1}\cdots f_r^{a_r}=$$
$$\big(|a|+1\big)\cdot \sum_{|\beta|-|\alpha|=1}\prod_{i=1}^r\frac{(a_i)!}{(\beta_i)!(a_i-\alpha_i)!}
P_{\alpha,\beta}\big(|a|\big)\bullet f_1^{a_1+\beta_1-\alpha_1}\cdots
f_r^{a_r+\beta_r-\alpha_r},$$
where the sum is over all $\alpha,\beta\in\Z_{\geq 0}^r$ with $\beta|-|\alpha|=1$ and such that $\alpha_i\leq a_i$ for all $i$. 
Since it is clear that $g$ is not invertible, we know that $(s+1)$ divides $b_g(s)$, with $\widetilde{b}_g(s)=b_g(s)/(s+1)$. 
It follows that $\widetilde{b}_g(s)$ is the monic polynomial of smallest degree such that we have $P_{\alpha,\beta}$ as above such that
for all $a=(a_1,\ldots,a_r)\in\Z_{\geq 0}^r$, we have
$$\widetilde{b}_g\big(|a|\big)f_1^{a_1}\cdots f_r^{a_r}=\sum_{|\beta|-|\alpha|=1}\prod_{i=1}^r\frac{(a_i)!}{(\beta_i)!(a_i-\alpha_i)!}
P_{\alpha,\beta}\big(|a|\big)\bullet f_1^{a_1+\beta_1-\alpha_1}\cdots
f_r^{a_r+\beta_r-\alpha_r}.$$
Equivalently, there are $P_{\alpha,\beta}\in \Gamma(X,{\mathcal D}_X)[s]$, for $\alpha,\beta\in\Z_{\geq 0}^r$ satisfying $|\beta|-|\alpha|=1$, with only finitely many nonzero, such that we have the equality
\begin{equation}\label{eq_proof6}
\widetilde{b}_g(s_1+\ldots+s_r)f_1^{s_1}\cdots f_r^{s_r}=\sum_{|\beta|-|\alpha|=1}\frac{\alpha !}{\beta !}\cdot \prod_{i=1}^r {{s_i}\choose {\alpha_i}}\cdot 
P_{\alpha,\beta}(s_1+\ldots+s_r)\bullet f_1^{s_1+\beta_1-\alpha_1}\cdots f_r^{s_r+\beta_r-\alpha_r}.
\end{equation}
Equivalently, $\widetilde{b}_g(s)$ is the monic polynomial of minimal degree such that
$\widetilde{b}_g(s_1+\ldots+s_r)f_1^{s_1}\cdots f_r^{s_r}$ lies in
$$\sum_{|\beta|-|\alpha|=1}\prod_{i=1}^r{{s_i}\choose {\alpha_i}}{\mathcal D}_X[s_1+\ldots+s_r]\bullet f_1^{s_1+\beta_1-\alpha_1}\cdots f_r^{s_r+\beta_r-\alpha_r}.$$
This sum can be rewritten as 
$$\sum_{|\gamma|=1}\sum_{\alpha}D_X[s_1+\ldots+s_r]\bullet\prod_{i=1}^r{{s_i}\choose{\alpha_i}}\cdot f_1^{s_1+\gamma_1}\cdots f_r^{s_r+\gamma_r},$$
where the first summation index runs over those $\gamma\in\Z_{\geq 0}^r$ such that $|\gamma|=1$ and the second summation index runs over
those $\alpha\in\Z_{\geq 0}^r$ such that $\alpha_i+\gamma_i\geq 0$ for all $i$. The polynomials ${{s_i}\choose{\alpha_i}}$ such that $\alpha_i+\gamma_i\geq 0$
give a basis of $\C[s_i]$ if $\gamma_i\geq 0$ and give a basis of ${{s_i}\choose {-\gamma_i}}\cdot\C[s_i]$ if $\gamma_i<0$. We thus conclude that
$\widetilde{b}_g(s)$ is the monic polynomial of smallest degree such that
$$\widetilde{b}_g(s_1+\ldots+s_r)f_1^{s_1}\cdots f_r^{s_r}\in \sum_{|\gamma|=1}{\mathcal D}_X[s_1,\ldots,s_r]\bullet 
\prod_{\gamma_i<0}{{s_i}\choose {-\gamma_i}}f_1^{s_1+\gamma_1}\cdots f_r^{s_r+\gamma_r},$$
hence it is equal to the Bernstein-Sato polynomial\footnote{This is not the definition of the Bernstein-Sato polynomial 
$b_{\fa}(s)$ in \cite{BMS}, but the definition is equivalent to this one, as explained in \cite[Section 2.10]{BMS}.} $b_{\fa}(s)$.
This completes the proof of the theorem.
\end{proof}

\begin{rmk}\label{rmk_existence}
Note that in the proof of Theorem~\ref{thm_description} we did not assume the existence of $b_{\fa}$, hence by the theorem, we can deduce
the existence of the Bernstein-Sato polynomial associated to $f_1,\ldots,f_r$ from the existence of $b_g(s)$. Furthermore, we see that 
$b_{\fa}(s)$ only depends on the ideal generated by $f_1,\ldots,f_r$ and not on these generators. Indeed, it is enough to show that if we consider 
$f_{r+1}=\sum_{i=1}^ra_if_i$ for some $a_1,\ldots,a_r\in\cO_X(X)$ and $h=\sum_{i=1}^{r+1}f_iy_i$, then $b_g(s)=b_h(s)$. 
Note that $h=\sum_{i=1}^rf_i(y_i+a_iy_{r+1})$. We have an automorphism of $X\times\A^{r+1}$ over $X$ which maps $y_{r+1}$ to $y_{r+1}$
and $y_i$ to $y_i+a_iy_{r+1}$ for $1\leq i\leq r$. Since this maps $g$ to $h$, it follows that $b_g(s)=b_h(s)$.
\end{rmk}

\begin{rmk}
The hypersurface $g=\sum_{i=1}^rf_iy_i$ also appeared in \cite{Javier}, where it was shown that its Milnor fibration
(at the origin) has trivial geometric monodromy and fiber homotopic to the complement of the germ defined by the ideal $(f_1,\ldots,f_r)$.
\end{rmk}

We can now deduce the first consequences of the theorem.

\begin{proof}[Proof of Corollary~\ref{cor1}]
It is shown in \cite[Theorem~2]{BMS} that the negative of the largest root of $b_{\fa}(s)$ is the log canonical threshold $\lct(\fa)$ of $\fa$.
Since $\widetilde{\alpha}_g$ is, by definition,  the negative of the largest root of $\widetilde{b}_g(s)$, the assertion follows from Theorem~\ref{thm_description}.
\end{proof}

\begin{proof}[Proof of Corollary~\ref{cor2}] 
Since $W$ is reduced and a complete intersection of pure codimension $r$, it follows from \cite[Theorem~4]{BMS} that $W$ has rational singularities if and only if $\lct(\fa)=r$
and $-r$ is a root of multiplicity 1 of ${b}_{\fa}(s)$. The assertion in the corollary thus follows from Theorem~\ref{thm_description}.
\end{proof}

\section{An application to the Strong Monodromy conjecture}

For a nice introduction to Igusa's zeta function we refer to \cite{Nicaise}.
We only recall here the definition of the $p$-adic absolute value and of the Haar measure on $\Z_p^n$. 
Let us denote by ${\rm ord}_p$ the $p$-adic valuation on $\Q_p$ (so that any element $u\in\Q_p$ can be
 written as $u=p^{{\rm ord}_p(u)}v$, with $v$ invertible in $\Z_p$). With this notation, if ${\rm ord}_p(u)=m$, then the $p$-adic absolute value of $u$
 is given by $|u|_p=\frac{1}{p^m}$. 

The Haar measure $\mu_p$ on $\Z_p^n$ is the unique translation-invariant measure such that $\mu_p(\Z_p^n)=1$. In particular, for 
every $u\in\Z_p^n$ and every positive integer $m$, we have
$$\mu_p(u+p^m\Z_p^n)=\frac{1}{p^{mn}}.$$
Note also that the Haar measure is multiplicative with respect to the Cartesian product of cylinders in $\Z_p^n\times\Z_p^r\simeq\Z_p^{n+r}$
(recall that a cylinder in $\Z_p^n$ is the inverse image of some set via a projection map $\Z_p^n\to (\Z/p^m\Z)^n$).

Given a nonzero $f\in\Z_p[x_1,\ldots,x_n]$, we denote by ${\rm ord}_f$ the function ${\rm ord}_p\circ f\colon \Z_p^n\to\Z_{\geq 0}$.
It then follows by definition that
\begin{equation}\label{eq1_Igusa}
Z_p(f;s)=\sum_{m\in\Z_{\geq 0}}\mu_p\big({\rm ord}_f^{-1}(m)\big)p^{-ms}.
\end{equation}
Similarly, if $\fa=(f_1,\ldots,f_r)$ is an ideal in $\Z_p[x_1,\ldots,x_n]$ and if we put ${\rm ord}_{\fa}=\min_{i=1}^r{\rm ord}_{f_i}$, then
\begin{equation}\label{eq2_Igusa}
Z_p(\fa;s)=\sum_{m\in\Z_{\geq 0}}\mu_p\big({\rm ord}_{\fa}^{-1}(m)\big)p^{-ms}.
\end{equation}
We can now prove the main result of this section.

\begin{proof}[Proof of Theorem~\ref{thm_Igusa}]
The key point is the computation of the $p$-adic measure of ${\rm ord}_g^{-1}(m)\subseteq\Z_p^{n+r}$ for each $m\geq 0$. 
Since $g=\sum_{i=1}^rf_iy_i$, it follows that if $(u,v_1,\ldots,v_r)\in\Z_p^{n+r}$ lies in ${\rm ord}_g^{-1}(m)$, then
$${\rm ord}_{\fa}(u)=\min_{i=1}^r{\rm ord}_{f_i}(u)\leq m.$$

Suppose now that $u\in\Z_p^n$ is such that $\min_{i=1}^r{\rm ord}_p(u_i)=d\leq m$. We want to describe the set $W_u(m)$ consisting of those
$v=(v_1,\ldots,v_r)\in\Z_p^r$ such that ${\rm ord}_p(u_1v_1+\ldots+u_rv_r)=m$. 
Suppose that $j$ is such that ${\rm ord}_p(u_j)=d$.
By assumption, we can write $u_i=t^du'_i$ for $1\leq i\leq r$ and $u'_i\in\Z_p$, with $u'_j$ invertible.
In this case, we have ${\rm ord}_p(u_1v_1+\ldots+u_rv_r)=m$ if and only if ${\rm ord}_p(u'_1v_1+\ldots+u'_rv_r)=m-d$. 
Since $u'_j$ is invertible, this means that $v_1,\ldots,\widehat{v_j},\ldots, v_r$ can be chosen arbitrarily and then the class of $v_j$ 
in $\Z/p^{m-d+1}\Z$ can take precisely $(p-1)$ values 
(and then every lift of this class satisfies the desired condition). We thus conclude that $W_u(m)\subseteq \Z_p^r$ is a cylinder whose $p$-adic measure 
is $\frac{p-1}{p^{m-d+1}}$. 

The projection $\Z_p^n\times\Z_p^r\to\Z_p^n$ onto the first component induces a map
$$\tau\colon {\rm ord}_g^{-1}(m)\to\bigsqcup_{d=0}^m {\rm ord}_{\fa}^{-1}(d).$$
If we decompose each ${\rm ord}_{\fa}^{-1}(d)$ as a disjoint union of cylinders such that on each of these cyclinders 
$\min_i{\rm ord}_{f_i}$ is achieved by some fixed $i$, then for every such cylinder $C\subseteq {\rm ord}_{\fa}^{-1}(d)$, the subset
$\tau^{-1}(C)\subseteq \Z_p^n\times\Z_p^r$ is a cylinder with
$$\mu_p\big(\tau^{-1}(C)\big)=\mu_p(C)\cdot \frac{p-1}{p^{m-d+1}}.$$
Therefore we have
$$\mu_p\big({\rm ord}_g^{-1}(m)\big)=\sum_{d=0}^m\mu_p\big({\rm ord}_{\fa}^{-1}(d)\big)\cdot\frac{p-1}{p^{m-d+1}}.$$
Using the formulas (\ref{eq1_Igusa}) and (\ref{eq2_Igusa}), we obtain
$$Z_p(g;s)=\sum_{m\geq 0}\frac{1}{p^{ms}}\cdot \sum_{d=0}^m\mu_p\big({\rm ord}_{\fa}^{-1}(d)\big)\cdot\frac{p-1}{p^{m-d+1}}$$
$$=\frac{p-1}{p}\cdot \sum_{d\geq 0}\frac{\mu_p\big({\rm ord}_{\fa}^{-1}(d)\big)}{p^{ds}}\cdot\sum_{m\geq d}\frac{1}{p^{(m-d)(s+1)}}
=\frac{1-p^{-1}}{1-p^{-(s+1)}}Z_p(\fa;s).$$
This gives the first assertion in the theorem.

The formula relating $Z_p(g;s)$ and $Z_p(\fa;s)$ shows that if we denote by
$n_p(g; \lambda)$ and $n_p(\fa; \lambda)$ the order of $\lambda$ as a pole of $Z_p(g; s)$ and $Z_p(\fa; s)$, respectively, then
$n_p(g; \lambda)=n_p(\fa; \lambda)$ for $\lambda\neq -1$; moreover, 
if $n_p(\fa;-1)\geq 1$, then $n_p(g;-1)=n_p(\fa;-1)+1$.
The second assertion in the theorem follows from this
and Theorem~\ref{thm_description}.
\end{proof}

\begin{rmk}\label{rmk_motivic}
For the sake of simplicity, we assumed in Theorem~\ref{thm_Igusa} that $\fa$ is an ideal in $\Z_p[x_1,\ldots,x_n]$. A similar formula holds, with the same proof,
if we assume that $f\in O_K[x_1,\ldots,x_n]$, where $O_K$ is the ring of integers of a $p$-adic field $K$. Moreover, the proof generalizes immediately to the case of the motivic zeta functions of Denef and Loeser \cite{DL}. In this case, we see that if $X$ is a smooth complex algebraic variety, $\fa$ is the coherent ideal
generated by $f_1,\ldots,f_r\in\cO_X(X)$, and $g=\sum_{i=1}^rf_iy_i$, then the motivic zeta functions $Z_{\rm mot}(g;s)$ and $Z_{\rm mot}(\fa; s)$ 
of $g$ and $\fa$, respectively, are related by the following formula
$$Z_{\rm mot}(g; s)=\frac{1-{\mathbf L}^{-1}}{1-{\mathbf L}^{-(s+1)}}Z_{\rm mot}(\fa; s).$$
\end{rmk}


\section*{References}
\begin{biblist}


\bib{Bernstein}{article}{
   author={Bern\v{s}te\u{i}n, I. N.},
   title={Modules over a ring of differential operators. An investigation of
   the fundamental solutions of equations with constant coefficients},
   journal={Funkcional. Anal. i Prilo\v{z}en.},
   volume={5},
   date={1971},
   number={2},
   pages={1--16},
}


\bib{Bjork}{book}{
       author={Bj{\"o}rk, J.},
       title={Analytic ${\mathcal D}$-modules and applications},  
       series={Mathematics and its Applications},  
       publisher={Kluwer Academic Publishers},
       date={1993},
}  



\bib{BMS}{article}{
   author={Budur, N.},
   author={Musta\c{t}\u{a}, M.},
   author={Saito, M.},
   title={Bernstein-Sato polynomials of arbitrary varieties},
   journal={Compos. Math.},
   volume={142},
   date={2006},
   number={3},
   pages={779--797},
}

\bib{DL}{article}{
   author={Denef, J.},
   author={Loeser, F.},
   title={Motivic Igusa zeta functions},
   journal={J. Algebraic Geom.},
   volume={7},
   date={1998},
   number={3},
   pages={505--537},
}

\bib{Javier}{article}{
   author={Fern\'{a}ndez de Bobadilla, J.},
   title={On homotopy types of complements of analytic sets and Milnor
   fibres},
   conference={
      title={Topology of algebraic varieties and singularities},
   },
   book={
      series={Contemp. Math.},
      volume={538},
      publisher={Amer. Math. Soc., Providence, RI},
   },
   date={2011},
   pages={363--367},
}

\bib{Igusa1}{article}{
   author={Igusa, J.},
   title={Complex powers and asymptotic expansions. I. Functions of certain
   types},
   note={Collection of articles dedicated to Helmut Hasse on his
   seventy-fifth birthday, II},
   journal={J. Reine Angew. Math.},
   volume={268/269},
   date={1974},
   pages={110--130},
}

\bib{Igusa2}{article}{
   author={Igusa, J.},
   title={Complex powers and asymptotic expansions. II. Asymptotic
   expansions},
   journal={J. Reine Angew. Math.},
   volume={278/279},
   date={1975},
   pages={307--321},
}

\bib{Kashiwara}{article}{
   author={Kashiwara, M.},
   title={$B$-functions and holonomic systems. Rationality of roots of
   $B$-functions},
   journal={Invent. Math.},
   volume={38},
   date={1976/77},
   number={1},
   pages={33--53},
}


\bib{Kollar}{article}{
   author={Koll\'ar, J.},
   title={Singularities of pairs},
   conference={
      title={Algebraic geometry---Santa Cruz 1995},
   },
   book={
      series={Proc. Sympos. Pure Math.},
      volume={62},
      publisher={Amer. Math. Soc., Providence, RI},
   },
   date={1997},
   pages={221--287},
}


\bib{Malgrange}{article}{
  author= {Malgrange, B.},
     title= {Polynomes de {B}ernstein-{S}ato et cohomologie \'evanescente},
 booktitle= {Analysis and topology on singular spaces, {II}, {III}
              ({L}uminy, 1981)},
    series = {Ast\'erisque},
    volume = {101},
    pages = {243--267},
 publisher = {Soc. Math. France, Paris},
      date = {1983},
      }



\bib{Nicaise}{article}{
   author={Nicaise, J.},
   title={An introduction to $p$-adic and motivic zeta functions and the
   monodromy conjecture},
   conference={
      title={Algebraic and analytic aspects of zeta functions and
      $L$-functions},
   },
   book={
      series={MSJ Mem.},
      volume={21},
      publisher={Math. Soc. Japan, Tokyo},
   },
   date={2010},
   pages={141--166},
}

\bib{Veys}{article}{
   author={Veys, W.},
   author={Z\'{u}\~{n}iga-Galindo, W. A.},
   title={Zeta functions for analytic mappings, log-principalization of
   ideals, and Newton polyhedra},
   journal={Trans. Amer. Math. Soc.},
   volume={360},
   date={2008},
   number={4},
   pages={2205--2227},
}




\end{biblist}
\end{document}